\documentclass[12pt]{amsart}
\usepackage[a4paper, total={156mm, 224mm}, centering]{geometry}
\usepackage{amsmath,amssymb,mathtools}
\usepackage{enumitem}
\usepackage{hyperref,hyphenat}
\hypersetup{hidelinks}

\numberwithin{equation}{section}

\newtheorem{theorem}{Theorem}[section]
\newtheorem{proposition}[theorem]{Proposition}
\newtheorem{lemma}[theorem]{Lemma}
\newtheorem{corollary}[theorem]{Corollary}
\theoremstyle{remark}
\newtheorem{remark}[theorem]{Remark}
\theoremstyle{plain}

\newcommand{\Z}{\mathbb{Z}}
\newcommand{\Q}{\mathbb{Q}}
\newcommand{\R}{\mathbb{R}}
\newcommand{\sO}{\mathcal{O}}
\newcommand{\sD}{\mathcal{D}}
\newcommand{\sA}{\mathcal{A}}
\newcommand{\bX}{\mathbf{X}}
\newcommand{\bY}{\mathbf{Y}}
\newcommand{\bZ}{\mathbf{Z}}
\newcommand{\by}{\mathbf{y}}
\newcommand{\bz}{\mathbf{z}}
\newcommand{\ba}{\mathbf{a}}
\newcommand{\bb}{\mathbf{b}}

\newcommand{\Sq}[1]{\operatorname{Sq}(#1)}
\newcommand{\Tr}{\operatorname{Tr}}
\newcommand{\Norm}{\operatorname{N}}

\title[$13$ unknowns and Lucas congruences]{$13$ unknowns over quadratic integer rings\\ and Lucas congruences}

\author{Geng-Rui Zhang}
\address{School of Mathematical Sciences, Peking University, Beijing 100871, China}
\email{grzhang@stu.pku.edu.cn, chibasei@163.com}
\date{August 31, 2026}


\begin{document}

\begin{abstract}
For every quadratic number field $K$, we prove a uniform $3$-unknown Diophantine definition of integer tuples in $\sO_K$, allowing finitely many polynomial nonvanishing conditions. This yields an effective $+3$ transfer principle and a $13$-unknown representation of every recursively enumerable integer relation. Consequently, there exists an absolute degree bound $D_0\geq1$ such that for every quadratic number field $K$, there is no algorithm that, given
\[
P(Y_1,\ldots,Y_{13})\in\Z[Y_1,\ldots,Y_{13}],\quad \deg P\leq D_0,
\]
decides whether $P=0$ has a solution in $\sO_K^{13}$.

The arithmetic input is a fourth-order Pell--Lucas congruence. It is a specialization of the norm-one Lucas multiplication formula, which yields exact valuations for the deviation of a Lucas quotient from its linear term, together with deviation criteria for Lucas--Wieferich and Wall--Sun--Sun primes. We also establish local surjectivity and $\ell$-adic density for second-order correction terms for norm-one Lucas sequences.
\end{abstract}

\maketitle
\tableofcontents

\section{Introduction}
Hilbert's tenth problem (HTP) over a commutative ring $R$ asks whether there is an algorithm deciding the solvability in $R$ of polynomial equations with coefficients in $R$. For the rings of integers of quadratic number fields, undecidability was proved by Denef \cite{Denef75}. Koymans--Pagano \cite{KP26} and Alp\"oge--Bhargava--Ho--Shnidman \cite{ABHS26} independently proved that HTP is undecidable over the ring of integers of every number field.

We study the quantitative problem for quadratic fields: how many unknowns suffice for undecidability when the input polynomial has coefficients in $\Z$? This question is related to quantitative Diophantine definability and existential rank; see, for example, \cite{DDF21}. Recent work of Bayer--David--Hassler--Matiyasevich--Schleicher studies universal unknown--degree pairs over $\Z$, together with a parallel formal-verification program \cite{BayerEtAl25,BayerDavid25}.

For the ring of Gaussian integers $\Z[i]$, Matiyasevich and Sun obtained a $20$-unknown bound \cite{MS27}, subsequently reduced to $18$ by Ding--Li \cite{DingLi26}. Sun \cite[Theorems~1.1 and~1.2]{Sun26} then proved that, for an arbitrary quadratic number field $K$, solvability over $\sO_K$ is undecidable for integer-coefficient polynomial equations in $16$ unknowns when $K$ is imaginary and in $15$ unknowns when $K$ is real.

We prove a uniform $13$-unknown bound together with an absolute degree bound:
\begin{theorem}\label{thm:bd13}
There exists an integer $D_0\geq1$ such that for every quadratic number field $K$, there is no algorithm that, given
\[
P(Y_1,\ldots,Y_{13})\in\Z[Y_1,\ldots,Y_{13}],\quad \deg P\leq D_0,
\]
decides whether $P=0$ has a solution in $\sO_K^{13}$.
\end{theorem}

The count follows from a fixed-cost transfer theorem: integer tuples, together with finitely many polynomial equalities and nonvanishing conditions, can be transferred to $\sO_K$ using $3$ additional unknowns; this number is independent of the tuple length and of the number of conditions.

\subsection{$3$-unknown transfer}
The main structural results of the paper are two fixed-cost integrality theorems.
\begin{theorem}[Uniform $3$-unknown integrality]\label{thm:uniform3}
Let $K$ be a quadratic number field and let $n\in\Z_{\geq0}$. For every
$R(\bX)\in\Z[X_1,\ldots,X_n]$, there is a polynomial
\[
\Phi_{K,n,R}(\bX,U,V,J)\in\Z[X_1,\ldots,X_n,U,V,J]
\]
such that, for every $\ba\in\sO_K^n$,
\[
\exists u,v,j\in\sO_K\colon\ \Phi_{K,n,R}(\ba,u,v,j)=0\quad\Longleftrightarrow\quad\ba\in\Z^n\text{ and }R(\ba)\neq0.
\]
Moreover, the polynomial $\Phi_{K,n,R}$ can be constructed effectively from $(K,n,R)$.
\end{theorem}
The theorem is uniform in $n$: exactly $3$ additional unknowns suffice for every $n$. It yields the following finite-system transfer principle. We use the convention that $\Z^0$ and $\sO_K^0$ consist of a single empty tuple.

\begin{theorem}[Finite-system $+3$ transfer]\label{thm:plus3}
Let $K$ be a quadratic number field, and let $r,m,s,t\in\Z_{\geq0}$. Let
\[
\bX=(X_1,\ldots,X_r)\quad\text{and}\quad\bY=(Y_1,\ldots,Y_m)
\]
be unknowns. Given
\[
P_1,\ldots,P_s,R_1,\ldots,R_t\in\Z[\bX,\bY],
\]
there is an effectively constructible polynomial
\[
H(\bX,\bY,U,V,J)\in\Z[\bX,\bY,U,V,J]
\]
such that, for every $\ba\in\sO_K^r$,
\[
\begin{split}
&\exists(\bb,u,v,j)\in\sO_K^{m+3}\colon\ H(\ba,\bb,u,v,j)=0\\
\Longleftrightarrow\hspace{2mm}&\ba\in\Z^r\quad\text{and}\quad\exists\bb\in\Z^m:\
P_1(\ba,\bb)=\cdots=P_s(\ba,\bb)=0,\\
&\hspace{47.3mm}R_1(\ba,\bb)\cdots R_t(\ba,\bb)\neq0,
\end{split}
\]
where the equality condition is vacuous if $s=0$, and the product is $1$ if $t=0$. Therefore, finitely many polynomial equalities and nonvanishing conditions over $\Z$ transfer to $\sO_K$ at the cost of $3$ additional unknowns.
\end{theorem}
Theorems~\ref{thm:uniform3} and~\ref{thm:plus3} still hold if the input polynomial(s) are allowed to have coefficients in $\sO_K$; see Remark~\ref{rmk:OKcoeff}.

The proofs use the following ingredients. Proposition~\ref{prop:pell-star} is the common arithmetic tool. Its fourth-order congruence determines a normalized Pell quotient modulo the square of a Pell coordinate, and Corollary~\ref{cor:reusex} uses the same auxiliary unknown to enforce a divisibility condition that encodes nonvanishing. Lemma~\ref{lem:F} by Sun combines two square conditions and this divisibility condition into one equation.

For Theorem~\ref{thm:uniform3}, the construction differs between the real and imaginary cases. If $K=\Q(\sqrt{-d})$ is imaginary, the theorem of Matiyasevich--Sun (Theorem~\ref{thm:MS}) encodes an arbitrary tuple into two quantities, and Sun's Pell rigidity (Lemma~\ref{lem:imagPell})
\[
u^2-3v^2=1,\quad u,v\in\sO_K\quad\Longrightarrow\quad u,v\in\Z
\]
establishes their integrality. Our Proposition~\ref{prop:imag-two-element} provides a two-element test for integers. If $K=\Q(\sqrt{d})$ is real, Denef \cite{Denef75} (Theorem~\ref{thm:Denef}) proved a square-rationality criterion. The real packing theorem (Theorem~\ref{thm:realpacking}) based on Sun's work, together with our trace--norm criterion in Proposition~\ref{prop:real-rationality}, uses Denef's square-rationality conclusion to deduce integrality.

After establishing Theorem~\ref{thm:uniform3}, Theorem~\ref{thm:plus3} follows by combining the finitely many conditions and applying Theorem~\ref{thm:uniform3} to the full tuple. The nonvanishing conditions are replaced by their product, the equalities are combined into one equation by a sum of squares in the real case, and by repeated use of the anisotropic form $A^2+(d+1)B^2$ of Sun (Lemma~\ref{lem:anisotropic}) in the imaginary case. Then Theorem~\ref{thm:uniform3} is applied to the full tuple. No additional unknowns are introduced.

\medskip

For Theorem~\ref{thm:bd13}, we need the following result of Sun \cite[Theorem~1.1(ii)]{Sun21}:
\begin{theorem}[Sun]\label{thm:source}
	For every recursively enumerable (r.e.) set $\sA\subseteq\Z_{\geq0}$, there is a polynomial $Q_{\sA}\in\Z[T,Z_1,\ldots,Z_{10}]$ such that for every $a\in\Z_{\geq0}$,
	\[
	a\in\sA\quad\Longleftrightarrow\quad\exists z_1,\ldots,z_{10}\in\Z\colon\ Q_{\sA}(a,z_1,\ldots,z_{10})=0,\ z_{10}\neq0.
	\]
	Consequently, there is no algorithm that, given
	\[
	P(Z_1,\ldots,Z_{10})\in\Z[Z_1,\ldots,Z_{10}],
	\]
	decides whether there exist $z_1,\ldots,z_{10}\in\Z$ such that
	\[
	P(z_1,\ldots,z_{10})=0,\quad z_{10}\neq0.
	\]
\end{theorem}
For basics of recursive theory, we refer the reader to \cite{Cutland}. Nonrecursive r.e.\ subsets of $\Z_{\geq0}$ exist; see \cite[\S~7.4]{Cutland}.

Applying Theorem~\ref{thm:plus3} to Theorem~\ref{thm:source} gives Theorem~\ref{thm:re13}: every r.e.\ relation on $\Z^r$ has a representation over $\sO_K$ with $13$ unknowns, uniformly in $r$. Fixing one nonrecursive r.e.\ set $\sA\subseteq\Z_{\geq0}$ then yields Theorem~\ref{thm:bd13}; the degree $D_0$ is uniformly bounded because in the explicit construction, the degrees of polynomials only depend on $\deg(Q_{\sA})$ and the signature of $K$.

\subsection{Toward $12$ unknowns}
The unknown count underlying Theorem~\ref{thm:bd13} factors as $13=10+3$, so there are two possible routes to $12$.

The first is to improve the source theorem over $\Z$, Theorem~\ref{thm:source}. Consequently, a $9$-unknown undecidability source theorem over $\Z$ would yield a $12$-unknown bound over $\sO_K$ via Theorem~\ref{thm:plus3}. Sun \cite[Theorem~1.1(i)]{Sun21} proved $9$-unknown representations of r.e.\ sets when one unknown is constrained to be nonnegative. Another possible approach is to reuse the nonnegativity unknown inside the transfer construction; the present construction does not achieve this.

The second route is to reduce the $3$-unknown cost in Theorem~\ref{thm:uniform3}. In both cases, its proof uses two Pell-square conditions. In the imaginary case, they force two mixed expressions to be integral. In the real case, their quotient supplies the rationality required by Proposition~\ref{prop:real-rationality}. The relation-combining polynomial $F$ requires the third auxiliary unknown; see Lemma~\ref{lem:F}. A uniform $2$-unknown definition of
\[
\ba\in\Z^n,\quad R(\ba)\neq0
\]
inside $\sO_K$ would therefore give a $12$-unknown theorem.

\subsection{Applications}
Section~\ref{sec:logic-pf} proves the transfer theorems and the $13$-unknown r.e.\ representation theorem (Theorem~\ref{thm:re13}); the latter yields Theorem~\ref{thm:bd13}. Then we transfer the explicit $11$-unknown universal pair $(11,\Delta_{\Z}\approx1.68105\cdot10^{63})$ over $\Z$ of Bayer et al.\ \cite{BayerEtAl25,BayerDavid25} to a $14$-unknown universal pair $(14,2\Delta_{\Z})$ over every quadratic integer ring. Corollary~\ref{cor:explicit-bounded} gives the resulting $14$-unknown undecidability statement for all quadratic integer rings, with a uniform explicit degree bound $2\Delta_{\Z}$.

Section~\ref{sec:pure-apps} studies the Lucas expansion behind Proposition~\ref{prop:pell-star}. Proposition~\ref{prop:lucas-all} and Corollary~\ref{cor:lucas4} present all-orders and fourth-order expansion formulas. Corollary~\ref{cor:deviation-vp} then isolates the finer arithmetic information used here: the exact $\ell$-adic valuation of the deviation of $u_{nk}/u_n$ from its linear term. Corollaries~\ref{cor:wieferich-deviation} and~\ref{cor:fib-deviation} express Lucas--Wieferich and Wall--Sun--Sun conditions through these deviations, while Theorem~\ref{thm:lucas-surj} and Corollary~\ref{cor:padic-density} show local surjectivity and $\ell$-adic density for normalized second-order correction terms.

\subsection*{Structure of the paper}
In \S~\ref{sec:prel}, we recall the relation-combining polynomial of Sun and develop the Pell congruence. In \S~\ref{sec:quad-input}, we establish the integrality criteria for imaginary and real quadratic fields. In \S~\ref{sec:logic-pf}, we complete the proofs of Theorems~\ref{thm:uniform3}, \ref{thm:plus3}, \ref{thm:re13}, and \ref{thm:bd13}; then we derive the universal-pair and the explicit bounded-degree consequence. In \S~\ref{sec:pure-apps}, we present the relevant Lucas multiplication expansion, exact linearization-deviation valuations, and Wieferich deviation criteria, and prove the local surjectivity and $\ell$-adic density results.

\section{Relation combining and Pell congruences}\label{sec:prel}
Throughout, if $A$ is a commutative ring and $a,b\in A$, then $a\mid b$ in $A$ means that $b=ac$ for some $c\in A$. For a commutative ring $A$, write $\Sq{A}=\{r^2\colon r\in A\}$.

\subsection{Nonzero integers and relation combining}\label{sec:comb}
Define
\[
D(X)=(2X+1)(3X+1)\in\Z[X].
\]
\begin{lemma}\label{lem:DL}
For $s\in\Z$, one has
\[
s\neq0\quad\Longleftrightarrow\quad\exists w\in\Z\colon\ s\mid D(w)\text{ in }\Z.
\]
\end{lemma}

It is easy to see that for every $q\in\Z_{>0}$, $D(X)\equiv0\pmod{q}$ is solvable in $\Z$ (cf.~\cite[Problem~202]{Sierpinski70}), which gives Lemma~\ref{lem:DL}. A related $2$-unknown parametrization of nonzero integers is attributed to Tung; see \cite[Lemma~3.4]{MS27}.

We will also use the fact that, for every algebraic integer $x\in\overline{\Q}$,
\begin{equation}\label{eq:Dnonzero}
D(x)\neq0.
\end{equation}

Define
\begin{align*}
F(A_1,A_2,S,T,J)={}&(T-JS)^4-2(A_1+A_2)S^2(T-JS)^2 \\
& +(A_1-A_2)^2S^4\in\Z[A_1,A_2,S,T,J].
\end{align*}
We need the following relation-combining lemma of Sun \cite[Lemma~2.2]{Sun26}:
\begin{lemma}[Sun]\label{lem:F}
Let $K$ be a number field. Let $a_1,a_2,s,t\in\sO_K$ with $a_1\neq a_2$ and $t\neq0$. Then the following are equivalent:
\begin{enumerate}[label=\textup{(\roman*)}]
\item $a_1,a_2\in\Sq{\sO_K}$ and $s\mid t$ in $\sO_K$;
\item there exists $j\in\sO_K$ such that $F(a_1,a_2,s,t,j)=0$.
\end{enumerate}
\end{lemma}
Lemma~\ref{lem:F} extends to every integrally closed domain of characteristic $\neq2$.

\subsection{Fourth-order Pell congruence}
We need the following lemma, which follows from \cite[p.~27]{BW23}. See also \cite[Lemma~3.5]{Sun26}.
\begin{lemma}[Pell divisibility]\label{lem:pell-div}
Let $a\geq2$ be an integer, set $e=a^2-1$, and define $p_j,q_j\in\Z_{\geq0}$ by
\[
p_j+q_j\sqrt{e}=(a+\sqrt{e})^j\quad(j\in\Z_{\geq0}).
\]
For every $h\in\Z\setminus\{0\}$, there exists an integer $r\in\Z_{>0}$ such that for every $m\in\Z_{>0}$,
\[
h\mid q_{m r}\text{ in }\Z.
\]
\end{lemma}

\begin{proposition}[Fourth-order Pell congruence]\label{prop:pell-star}
Let $a\geq2$ be an integer, set $e=a^2-1$, and define $p_j,q_j\in\Z_{\geq0}$ by
\[
p_j+q_j\sqrt{e}=(a+\sqrt{e})^j\quad(j\in\Z_{\geq0}).
\]
Fix $n\in\Z_{>0}$. Define
\[
\gamma_e(Z)=\frac{eZ(Z^2-1)}{6}\in\Q[Z].
\]
Then the following two statements hold.
\begin{enumerate}[label=\textup{(\arabic*)}]
	\item
	For every positive odd integer $k$, one has $q_n\mid q_{nk}$ in $\Z$ and
	\[
	\frac{q_{nk}}{q_n}\equiv k+\gamma_e(k)q_n^2\pmod{q_n^4},
	\]
	where $\gamma_e(k)\in\Z_{\geq0}$. In particular,
	\[
	\frac{q_{nk}}{q_n}\equiv k\pmod{q_n^2}.
	\]
	
	\item
	Let $s\in\Z\setminus\{0\}$ satisfy $s\mid q_n$ in $\Z$, let $t\in\Z$ be odd, fix $\varepsilon\in\{\pm1\}$, and fix a residue class $c\pmod{s}$. Then there exist infinitely many $r\in\Z$ such that $k:=t+r q_n^2$ is a positive odd integer and, with
	\[
	\lambda_k:=\frac{q_{nk}}{q_n}\in\Z_{>0}\quad\text{and}\quad
	x:=\varepsilon\frac{\lambda_k-t}{q_n^2}=\varepsilon\left(\frac{\lambda_k-k}{q_n^2}+r\right)\in\Z,
	\]
	one has
	\[
	x\equiv c\pmod{s}.
	\]
\end{enumerate}
\end{proposition}
\begin{proof}
\textup{(1)}. The case $k=1$ is trivial. Suppose that $k\geq3$ is an odd integer. The sequence $(q_j)_j$ is the Lucas sequence $\left(U_j(2a,1)\right)_j$, while $2p_j=V_j(2a,1)$. By the classical Lucas multiplication formula \cite[(2.28)]{BW23}, we have
\begin{equation}\label{eq:qexpand}
	\frac{q_{nk}}{q_n}
	=\sum_{j=0}^{(k-1)/2}\binom{k}{2j+1}e^j
	p_n^{k-2j-1}q_n^{2j}.
\end{equation}
In particular, $q_n\mid q_{nk}$ in $\Z$.

Only the terms $0\leq j\leq1$ in \eqref{eq:qexpand} contribute modulo $q_n^4$. Moreover, from $p_n^2-eq_n^2=1$ we obtain
\[
p_n^{k-1}=(1+eq_n^2)^{(k-1)/2}\equiv1+\frac{k-1}{2}eq_n^2\pmod{q_n^4}\quad\text{and}\quad p_n^{k-3}\equiv1\pmod{q_n^2}.
\]
Thus,
\[
\frac{q_{nk}}{q_n}\equiv k\left(1+\frac{k-1}{2}eq_n^2\right)+\binom{k}{3}eq_n^2=k+\gamma_e(k)q_n^2\pmod{q_n^4}.
\]

\smallskip

\textup{(2)}.
Choose $r$ modulo $s$ so that
\begin{equation}\label{eq:r-choice}
r\equiv\varepsilon c-\gamma_e(t)\pmod{s}.
\end{equation}
Then there are infinitely many representatives $r$ of this residue class for which
\[
k:=t+r q_n^2
\]
is a positive odd integer. Indeed, if $q_n$ is even, then $q_n^2r$ is even for every $r\in\Z$, so $k$ is odd because $t$ is odd. Fix any representative $r_0$ of the residue class in \eqref{eq:r-choice}. Then
\[
r=r_0+m|s|\qquad(m\in\Z_{\geq0})
\]
remains in this residue class, and for all sufficiently large $m$, the corresponding $k$ is positive. Suppose now that $q_n$ is odd. Then $s$ is also odd because $s\mid q_n$. Hence, exactly one of $r_0$ and $r_0+s$ is even; denote this even representative by $r_1$. Then
\[
r=r_1+2m|s|\qquad(m\in\Z_{\geq0})
\]
remains in the prescribed residue class and is even. Since both $t$ and $q_n^2$ are odd, the corresponding integer $k=t+r q_n^2$ is odd, and it is positive for all sufficiently large $m$. Thus, in either case, there are infinitely many desired choices of $r$.

Fix any such $r$. By part~\textup{(1)}, we have $\lambda_k=k+\gamma_e(k)q_n^2+bq_n^4$ for some $b\in\Z$. Hence
\begin{equation}\label{eq:rescalc}
\frac{\lambda_k-t}{q_n^2}=\frac{\lambda_k-(k-rq_n^2)}{q_n^2}=r+\gamma_e(k)+bq_n^2.
\end{equation}
Note that
\[
6\left(\gamma_e(k)-\gamma_e(t)\right)=e(k-t)(k^2+kt+t^2-1).
\]
If $v_\ell(s)=f$ with $\ell$ a prime and $f\geq1$, then
$v_\ell(q_n)\geq f$ because $s\mid q_n$, and
\[
v_\ell\left(\gamma_e(k)-\gamma_e(t)\right)\geq v_\ell(k-t)-v_\ell(6)=v_\ell(rq_n^2)-v_\ell(6)\geq2f-1\geq f.
\]
Therefore,
\[
\gamma_e(k)\equiv\gamma_e(t)\pmod{s}.
\]
Using \eqref{eq:rescalc}, $s\mid q_n$, and \eqref{eq:r-choice}, we obtain
\[
\frac{\lambda_k-t}{q_n^2}
\equiv r+\gamma_e(t)
\equiv\varepsilon c\pmod{s}.
\]
Multiplying by $\varepsilon$ shows that the integer $x$ satisfies $x\equiv c\pmod{s}$. Since the preceding argument applies to each of the infinitely many choices of $r$, the conclusion follows.
\end{proof}

Proposition~\ref{prop:pell-star}\textup{(2)} yields the following corollary:
\begin{corollary}\label{cor:reusex}
Let $a\geq2$ be an integer, set $e=a^2-1$, and define $p_j,q_j\in\Z_{\geq0}$ by
\[
p_j+q_j\sqrt{e}=(a+\sqrt{e})^j\quad(j\in\Z_{\geq0}).
\]
Fix $n\in\Z_{>0}$. Let $s\in\Z\setminus\{0\}$ satisfy $s\mid q_n$ in $\Z$, let $t\in\Z$ be odd, and fix a sign $\varepsilon\in\{\pm1\}$. Then there exist infinitely many pairs $(r,x)\in\Z^2$ such that $k=t+r q_n^2$ is a positive odd integer,
\[
x=\varepsilon\frac{q_{nk}/q_n-t}{q_n^2},\quad\text{and}\quad s\mid D(x).
\]
\end{corollary}
\begin{proof}
By Lemma~\ref{lem:DL}, choose $c\in\Z$ such that $s\mid D(c)$. Apply Proposition~\ref{prop:pell-star}\textup{(2)} to the residue class $c\pmod{s}$. It gives infinitely many $r\in\Z$ such that $k=t+r q_n^2$ is a positive odd integer and, with
\[
x=\varepsilon\frac{q_{nk}/q_n-t}{q_n^2}\in\Z,
\]
one has $x\equiv c\pmod{s}$. Hence
\[
D(x)\equiv D(c)\equiv0\pmod{s},
\]
and therefore $s\mid D(x)$.
\end{proof}

\section{Quadratic integrality criteria}\label{sec:quad-input}
The proof of the $3$-unknown theorem requires different inputs for the imaginary and real quadratic fields. We collect them in this section.

\subsection{Imaginary quadratic fields}
Fix a squarefree integer $d\in\Z_{>0}$ and set $K=\Q(\sqrt{-d})$.

We use two lemmas of Sun \cite[Lemmas~2.3 and~2.4]{Sun26}.
\begin{lemma}[Sun]\label{lem:anisotropic}
For $u,v\in\sO_K$,
\[
u=v=0\quad\Longleftrightarrow\quad u^2+(d+1)v^2=0.
\]
\end{lemma}
Clearly, Lemma~\ref{lem:anisotropic} also works for $u,v\in K$.

\begin{lemma}[Sun]\label{lem:imagPell}
If $u,v\in\sO_K$ satisfy $u^2-3v^2=1$, then $u,v\in\Z$.
\end{lemma}

Define the polynomial
\begin{equation}\label{eq:hdef}
H_0(X)=2(2X+1)^2+1\in\Z[X].
\end{equation}
For every $z\in\sO_K$, since $2z+1\in\sO_K$ and the roots of $2Z^2+1=0$ are not algebraic integers, we have $H_0(z)\neq0$.

\begin{proposition}[Two-element integer test]\label{prop:imag-two-element}
Let $z_1,z_2\in\sO_K$, and set
\[
h=H_0(z_1)H_0(z_2)\in\sO_K\setminus\{0\}.
\]
Then the following three statements are equivalent.
\begin{enumerate}[label=\textup{(\arabic*)}]
\item $z_1,z_2\in\Z$.

\item For every $t\in\Z\setminus\{0\}$ there exist $v,x\in\Z$ such that, on setting
\[
s=vt,\quad y=sh,\quad b=3y^2+1,\quad c=3y^2\left(2(z_1+yz_2)+1-xy^2\right)^2+1,
\]
one has $b,c\in\Sq{\Z}$ and $s\mid D(x)$ in $\Z$.

\item There exist $s\in\sO_K\setminus\{0\}$ and $x\in\sO_K$ such that, on setting
\[
y=sh,\quad b=3y^2+1,\quad c=3y^2\left(2(z_1+yz_2)+1-xy^2\right)^2+1,
\]
one has $b,c\in\Sq{\sO_K}$.
\end{enumerate}
\end{proposition}
\begin{proof}
We prove \textup{(3)}$\Rightarrow$\textup{(1)}, \textup{(1)}$\Rightarrow$\textup{(2)}, and finally \textup{(2)}$\Rightarrow$\textup{(3)}.

Assume \textup{(3)}. Write $b=u^2$ with $u\in\sO_K$. Then $u^2-3y^2=1$, so Lemma~\ref{lem:imagPell} implies $u,y\in\Z$. Note that $y=sh\neq0$.

Every nonzero algebraic integer $\alpha$ in an imaginary quadratic field satisfies
\[
|\alpha|^2=\Norm_{K/\Q}(\alpha)\in\Z_{>0},
\]
and hence $|\alpha|\geq1$. Then
\begin{equation}\label{eq:Hbound}
|H_0(z_j)|\leq|h|=\frac{|y|}{|s|}\leq|y|\quad(j=1,2).
\end{equation}
We claim that $|y|\geq2$. If $|y|=1$, then \eqref{eq:Hbound} forces both $H_0(z_j)$ to be units. The unit group of an imaginary quadratic integer ring is $\{\pm1\}$, except for $d=1,3$; see \cite[Proposition~5.8]{Kha22}. In the exceptional cases $d=1,3$, it is easy to see that the additional roots of unity are not congruent to $1$ modulo $2\sO_K$. Since $H_0(Z)\equiv1\pmod{2\sO_K}$, we conclude that $H_0(z_j)\in\{\pm1\}$ for $1\leq j\leq 2$. The equation $H_0(z)=1$ gives $2z+1=0$, impossible for an algebraic integer; the equation $H_0(z)=-1$ gives $(2z+1)^2=-1$, which can occur in $K$ only for $d=1$, and then $z=(-1\pm i)/2\notin\Z[i]$. Thus, $|y|\geq2$.

Since $u^2-3y^2=1$ with $u\in\Z$ and $y\in\Z\setminus\{0,\pm1\}$, we see that
\[
\rho:=|y|\geq4.
\]
By \eqref{eq:Hbound}, for $1\leq j\leq2$ we obtain
\[
2|2z_j+1|^2-1\leq|H_0(z_j)|\leq\rho,
\]
so
\begin{equation}\label{eq:zsize}
|2z_j+1|\leq\sqrt{\frac{\rho+1}{2}}.
\end{equation}
Set
\[
\kappa=
\begin{cases}
2,&d\equiv3\pmod{4},\\
1,&d\not\equiv3\pmod{4}.
\end{cases}
\]
Write
\[
x=\frac{\alpha+\beta\sqrt{-d}}{\kappa}\quad\text{and}\quad z_j=\frac{\alpha_j+\beta_j\sqrt{-d}}{\kappa},
\]
with $\alpha,\beta,\alpha_j,\beta_j\in\Z$, $j=1,2$. If $\kappa=2$, we have
\[
\alpha\equiv\beta\pmod{2}\quad\text{and}\quad\alpha_j\equiv\beta_j\pmod{2}\quad(j=1,2).
\]
The imaginary part of $2z_j+1$ has absolute value $2|\beta_j|\sqrt{d}/\kappa$; then \eqref{eq:zsize} yields
\begin{equation}\label{eq:bjbound}
|\beta_j|\leq\kappa\sqrt{\frac{\rho+1}{8d}}\quad(j=1,2).
\end{equation}

Now write $c=w^2$ with $w\in\sO_K$ and set
\[
q=2(z_1+yz_2)+1-xy^2\in\sO_K.
\]
Then
\[
w^2-3(yq)^2=c-3y^2q^2=1,
\]
and Lemma~\ref{lem:imagPell} implies $w,yq\in\Z$. Since $y\in\Z\setminus\{0\}$ and $q\in\sO_K$, it follows that
\[
q=\frac{yq}{y}\in\Q\cap\sO_K=\Z.
\]
The coefficient of $\sqrt{-d}$ in $q$ therefore vanishes:
\begin{equation}\label{eq:coeffimag}
2\beta_1+2y\beta_2-\rho^2\beta=0.
\end{equation}
If $\beta\neq0$, then \eqref{eq:bjbound} and \eqref{eq:coeffimag} give
\[
\rho^2\leq\rho^2|\beta|=|2\beta_1+2y\beta_2|\leq2|\beta_1|+2\rho|\beta_2|\leq\frac{\kappa(\rho+1)^{3/2}}{\sqrt{2d}}\leq\sqrt{2}(\rho+1)^{3/2}.
\]
This is impossible for $\rho\geq4$; hence $\beta=0$. Equation \eqref{eq:coeffimag} becomes $\beta_1=-y\beta_2$. But \eqref{eq:bjbound} implies $|\beta_1|<\rho$. If $\beta_2\neq0$, then $|\beta_1|=\rho|\beta_2|\geq\rho$, a contradiction. Thus, $\beta_1=\beta_2=0$, and then $z_1,z_2\in\Z=\sO_K\cap\Q$. This establishes \textup{(3)}$\Rightarrow$\textup{(1)}.

\smallskip

Now assume \textup{(1)} and fix an arbitrary $t\in\Z\setminus\{0\}$. Note that $H_0(z_1)$, $H_0(z_2)$ and $h$ are positive odd integers. Use the Pell sequence
\[
p_j+q_j\sqrt3=(2+\sqrt3)^j\quad(j\in\Z_{\geq0}).
\]
By Lemma~\ref{lem:pell-div}, take $n\in\Z_{>0}$ such that $ht\mid q_n$. Set
\[
v=\frac{q_n}{ht},\quad s=vt=\frac{q_n}{h},\quad\text{and}\quad y=sh=q_n.
\]
Then
\[
b:=3y^2+1=3q_n^2+1=p_n^2
\]
is a square in $\Z$. Define
\[
t_0:=2(z_1+yz_2)+1,
\]
which is an odd integer. Set
\[
\varepsilon=\operatorname{sgn}(t_0)\in\{\pm1\}\quad\text{and}\quad\tau=|t_0|\in\Z_{>0}.
\]
Apply Corollary~\ref{cor:reusex} with $(a,e)=(2,3)$, $(n,s)$ as above, the prescribed positive odd integer $\tau$, and the sign $-\varepsilon$. We obtain a positive odd integer $k$ and an integer $x$ such that
\[
x=-\varepsilon\frac{q_{nk}/q_n-\tau}{q_n^2}\quad\text{and}\quad s\mid D(x).
\]
Consequently,
\[
t_0-xy^2=\varepsilon\tau+\varepsilon\left(\frac{q_{nk}}{q_n}-\tau\right)=\varepsilon\frac{q_{nk}}{q_n},
\]
and therefore
\[
c:=3y^2(t_0-xy^2)^2+1=3q_n^2\left(\varepsilon\frac{q_{nk}}{q_n}\right)^2+1=3q_{nk}^2+1=p_{nk}^2
\]
is a square in $\Z$. This proves \textup{(1)}$\Rightarrow$\textup{(2)}.

\smallskip

Finally, assume \textup{(2)} and take $t=1$ in \textup{(2)}. Then the corresponding integers $v=s$ and $x$ give $s\mid D(x)$. It suffices to show that $s\neq0$, which follows from $s\mid D(x)$ and \eqref{eq:Dnonzero}. Thus \textup{(2)}$\Rightarrow$\textup{(3)}, and all three statements are equivalent.
\end{proof}

We also need the following result of Matiyasevich and Sun \cite[Theorem~1.2]{MS27}:
\begin{theorem}[Matiyasevich--Sun]\label{thm:MS}
Let $n,d\geq1$ be integers with $d$ squarefree. Let $K=\Q(\sqrt{-d})$ and $x_1,\ldots,x_n\in\sO_K$. Set
\[
\eta=2\prod_{j=1}^n(3x_j+1).
\]
Then
\[
\eta+\sum_{j=1}^n\frac{x_j}{\eta^j}\in\Q\quad\Longleftrightarrow\quad x_1,\ldots,x_n\in\Z.
\]
\end{theorem}

\subsection{Real quadratic fields}\label{sec:real}
Fix a squarefree integer $d\geq2$, set $K=\Q(\sqrt{d})$, and let $\sigma_1,\sigma_2:K\hookrightarrow\R$ be the two real embeddings of $K$.

Choose integers $a_0\geq2$ and $b_0\geq1$ such that $a_0^2-db_0^2=1$. Such a nontrivial positive integer solution exists by the classical Pell equation. Define
\begin{equation}\label{eq:real-ae}
a=a_0^2+db_0^2=2a_0^2-1,\quad b=2a_0b_0,\quad e=a^2-1=db^2.
\end{equation}
Then $a\geq7$ is odd and $e\geq48$ is even. Define $p_j,q_j\in\Z_{\geq0}$ by
\begin{equation*}
p_j+q_j\sqrt{e}=(a+\sqrt{e})^j\quad(j\in\Z_{\geq0}).
\end{equation*}

We use the following consequence of Denef \cite[Lemma~4]{Denef75}; see also \cite[Lemma~3.3(i)]{Sun26}.
\begin{theorem}[Denef]\label{thm:Denef}
If $u,v\in\sO_K$ satisfy $u^2-ev^2=1$, then $v^2\in\Z_{\geq0}$.
\end{theorem}

The following real packing theorem is based on Sun \cite[Theorem~3.1]{Sun26}:
\begin{theorem}[Real packing]\label{thm:realpacking}
Let $n\in\Z_{>0}$ and $x_1,\ldots,x_n\in\sO_K$. Set
\begin{equation}\label{eq:real-mw-value}
m=2\prod_{j=1}^n(x_j^2+1)\quad\text{and}\quad w=m^{n+1}+\sum_{j=1}^n x_jm^{n-j}.
\end{equation}
Then the following two statements hold.
\begin{enumerate}[label=\textup{(\arabic*)}]
\item\label{realpack-positive}$\sigma_1(w)>3$ and $\sigma_2(w)>3$.

\item For every $r\in\Z_{>0}$, we have $w^r\in\Q$ if and only if $x_1,\ldots,x_n\in\Z$.
\end{enumerate}
\end{theorem}
\begin{proof}
\textup{(1)}. For $i\in\{1,2\}$, write
\[
m_i=\sigma_i(m)=2\prod_{j=1}^n\left(\sigma_i(x_j)^2+1\right)>0\quad\text{and}\quad r_i=\sum_{j=1}^n\sigma_i(x_j)m_i^{n-j}.
\]
Since $t^2+1\geq2|t|$ for $t\in\R$, we have, for all $i$ and $j$,
\[
m_i\geq2\left(\sigma_i(x_j)^2+1\right)\geq4|\sigma_i(x_j)|.
\]
Also $m_i\geq2$. Consequently,
\[
|r_i|\leq\frac{m_i}{4}\sum_{j=1}^n m_i^{n-j}=\frac{m_i^n}{4}\sum_{r=0}^{n-1}m_i^{-r}<\frac{m_i^n}{2}.
\]
Therefore,
\[
\sigma_i(w)=m_i^{n+1}+r_i>m_i^n\left(m_i-\frac{1}{2}\right)\geq m_i^n\left(2-\frac{1}{2}\right)\geq2\cdot\frac{3}{2}=3.
\]

\smallskip

\textup{(2)}. For the elements $m,w$ in \eqref{eq:real-mw-value}, Sun's theorem \cite[Theorem~3.1]{Sun26} gives
\begin{equation}\label{eq:sun-realpacking}
w\in\Q\quad\Longleftrightarrow\quad x_1,\ldots,x_n\in\Z.
\end{equation}

Let $r\in\Z_{>0}$. If $x_1,\ldots,x_n\in\Z$, then $w\in\Z$, and hence $w^r\in\Q$.
Conversely, suppose that $w^r\in\Q$. Then
\[
\sigma_1(w)^r=w^r=\sigma_2(w)^r.
\]
By \ref{realpack-positive}, both $\sigma_1(w)$ and $\sigma_2(w)$ are strictly positive, so $\sigma_1(w)=\sigma_2(w)$ and therefore $w\in\Q$. Then \eqref{eq:sun-realpacking} yields $x_1,\ldots,x_n\in\Z$.
\end{proof}

\begin{proposition}\label{prop:real-rationality}
Let $\Delta_K$ be the absolute discriminant of $K$. Let $\nu\in\Z_{>0}$ and $t,q,x\in\sO_K$ satisfy
\begin{equation}\label{eq:realrat-data}
q^2\in\Q\quad\text{and}\quad q=t+\nu x.
\end{equation}
Suppose further that
\begin{equation}\label{eq:realrat-sharp}
\nu\nmid \Tr_{K/\Q}(t)\text{ in }\Z\quad\text{and}\quad|\sigma_1(t)-\sigma_2(t)|<\nu\sqrt{\Delta_K}.
\end{equation}
Then $t\in\Z$.
\end{proposition}
\begin{proof}
Write $q=\alpha+\beta\sqrt{d}$ with $\alpha,\beta\in\Q$. By \eqref{eq:realrat-data}, we obtain $\alpha\beta=0$. Then either $q\in\Q$ or $\Tr_{K/\Q}(q)=0$.

Suppose first that $\Tr_{K/\Q}(q)=0$. Taking traces in $q=t+\nu x$ yields
\[
\Tr_{K/\Q}(t)+\nu\Tr_{K/\Q}(x)=0.
\]
Since $x\in\sO_K$, its trace is an integer, whence $\nu\mid\Tr_{K/\Q}(t)$ in $\Z$, contrary to \eqref{eq:realrat-sharp}. Therefore $q\in\Q\cap\sO_K=\Z$.

Take the standard integral basis $1,\omega$ of $\sO_K$, so that
\[
\omega=
\begin{cases}
\sqrt{d},&d\not\equiv1\pmod{4},\\[2pt]
(1+\sqrt{d})/2,&d\equiv1\pmod{4},
\end{cases}
\quad
\sqrt{\Delta_K}=|\sigma_1(\omega)-\sigma_2(\omega)|.
\]
Write $t=\alpha_0+\beta_0\omega$ and $x=\gamma_0+\delta_0\omega$ with $\alpha_0,\beta_0,\gamma_0,\delta_0\in\Z$. Since $q=t+\nu x\in\Z$, comparison of the $\omega$-coefficients implies $\beta_0+\nu\delta_0=0$; hence $\nu\mid\beta_0$ in $\Z$. If $\beta_0\neq0$, then
\[
|\sigma_1(t)-\sigma_2(t)|=|\beta_0||\sigma_1(\omega)-\sigma_2(\omega)|=|\beta_0|\sqrt{\Delta_K}\geq\nu\sqrt{\Delta_K},
\]
contrary to \eqref{eq:realrat-sharp}. Thus $\beta_0=0$, and therefore $t\in\Z$.
\end{proof}

\section{Diophantine transfer and undecidability}\label{sec:logic-pf}
We now prove the transfer results stated in the introduction. The main theorem (Theorem~\ref{thm:bd13}) will follow from the r.e.\ representation theorem (Theorem~\ref{thm:re13}) in \S~\ref{sec:re}.

\subsection{Transfer principles}
\begin{proof}[Proof of Theorem~\ref{thm:uniform3}]
Fix $n\in\Z_{\geq0}$. Let $R(\bX)\in\Z[X_1,\ldots,X_n]$ be an arbitrary polynomial, where $\bX=(X_1,\ldots,X_n)$ are unknowns.

We first deal with the case $n=0$. In this case, $\Z[X_1,\ldots,X_n]=\Z$; so $R$ is a fixed integer, and the unique element of $\sO_K^0$ is the empty tuple, which also belongs to $\Z^0$. If $R\neq0$, set $\Phi_{K,0,R}(U,V,J)=0$. If $R=0$, set $\Phi_{K,0,R}(U,V,J)=1$. Then
\[
\exists u,v,j\in\sO_K\colon\ \Phi_{K,0,R}(u,v,j)=0\quad\Longleftrightarrow\quad R\neq0,
\]
which is exactly the required equivalence. This construction is clearly effective. It remains to consider $n\geq1$.

\smallskip

Suppose now $n\geq1$. We treat imaginary and real quadratic fields separately. 

\smallskip

\noindent\emph{Imaginary quadratic fields.}
Write $K=\Q(\sqrt{-d})$, where $d\geq1$ is a squarefree integer. Define the polynomials
\begin{align*}
L_0(\bX)&=2\prod_{i=1}^n(3X_i+1),\\
L_1(\bX)&=L_0(\bX)^n,\\
L_2(\bX)&=L_0(\bX)^{n+1}+\sum_{i=1}^nX_iL_0(\bX)^{n-i}.
\end{align*}
Theorem~\ref{thm:MS} gives, for every $\ba=(a_1,\ldots,a_n)\in\sO_K^n$,
\begin{equation}\label{eq:uniform-imag-MS}
L_1(\ba),L_2(\ba)\in\Z\quad\Longrightarrow\quad\frac{L_2(\ba)}{L_1(\ba)}=L_0(\ba)+\sum_{i=1}^n\frac{a_i}{L_0(\ba)^i}\in\Q\quad\Longrightarrow\quad\ba\in\Z^n.
\end{equation}
The converse implication $\ba\in\Z^n\Rightarrow L_1(\ba),L_2(\ba)\in\Z$ holds trivially. Note that $L_0(\ba)\neq0$ for every $\ba\in\sO_K^n$.

Use the polynomial $H_0$ from \eqref{eq:hdef}, and introduce the $3$ new
unknowns $U,V,J$. Set
\begin{align*}
H(\bX)&=H_0(L_1(\bX))H_0(L_2(\bX)),\\
S(\bX,U)&=U R(\bX),\\
Y(\bX,U)&=S(\bX,U)H(\bX),\\
B(\bX,U)&=3Y(\bX,U)^2+1,\\
C(\bX,U,V)&=3Y(\bX,U)^2\left(2\left(L_1(\bX)+Y(\bX,U)L_2(\bX)\right)+1-VY(\bX,U)^2\right)^2+1.
\end{align*}
Finally, set
\[
\Phi_{K,n,R}(\bX,U,V,J)=F\left(9B(\bX,U),C(\bX,U,V),S(\bX,U),D(V),J\right),
\]
where $F$ and $D$ are as in \S~\ref{sec:comb}.

For every evaluation $(\ba,u,v,j)=(a_1,\ldots,a_n,u,v,j)\in\sO_K^{n+3}$, write
\[
l_i=L_i(\ba)\ (i=0,1,2),\quad h=H(\ba),\quad s=S(\ba,u),\quad y=Y(\ba,u),
\]
\[
b=B(\ba,u),\quad c=C(\ba,u,v).
\]
Then
\[
9b\equiv0\pmod{3\sO_K}\quad\text{and}\quad c\equiv1\pmod{3\sO_K},
\]
so $9b\neq c$. Moreover, since $b\in\sO_K$ and $\sO_K$ is integrally closed,
\begin{equation}\label{eq:uniform9square}
9b\in\Sq{\sO_K}\quad\Longleftrightarrow\quad b\in\Sq{\sO_K}.
\end{equation}

Suppose first that $\ba\in\Z^n$ and $R(\ba)\neq0$. Then $l_1,l_2\in\Z$. By the direction \textup{(1)}$\Rightarrow$\textup{(2)} of Proposition~\ref{prop:imag-two-element}, there are $u,v\in\Z$ such that, for
$(z_1,z_2)=(l_1,l_2)$, $t=R(\ba)\neq0$, and $s=uR(\ba)$, one has
\[
b,c\in\Sq{\Z}\quad\text{and}\quad s\mid D(v)\text{ in }\Z.
\]
Then $9b$ and $c$ are squares in $\Z$. Lemma~\ref{lem:F} for $\Z$ shows that there exists $j\in\Z$ such
that $\Phi_{K,n,R}(\ba,u,v,j)=0$.

Conversely, for $(\ba,u,v,j)\in\sO_K^{n+3}$, suppose
$\Phi_{K,n,R}(\ba,u,v,j)=0$. Since $D(v)\neq0$, Lemma~\ref{lem:F} for $\sO_K$ implies
\[
9b,c\in\Sq{\sO_K}\quad\text{and}\quad s\mid D(v)\text{ in }\sO_K.
\]
Hence $s\neq0$, and we have $b\in\Sq{\sO_K}$ by \eqref{eq:uniform9square}. The direction \textup{(3)$\Rightarrow$(1)} of Proposition~\ref{prop:imag-two-element} yields
$l_1,l_2\in\Z$. Observe that $\ba\in\Z^n$ by \eqref{eq:uniform-imag-MS}. Finally, $R(\ba)\neq0$ because $s=uR(\ba)\neq0$.

\smallskip

\noindent\emph{Real quadratic fields.}
Write $K=\Q(\sqrt{d})$, where $d\geq2$ is squarefree. Define
\begin{align*}
M(\bX)&=2\prod_{i=1}^n(X_i^2+1),\\
W(\bX)&=M(\bX)^{n+1}+\sum_{i=1}^nX_iM(\bX)^{n-i},\\
T(\bX)&=2W(\bX)+1.
\end{align*}
For every $\ba\in\sO_K^n$, Theorem~\ref{thm:realpacking} \textup{(1)} implies
\begin{equation}\label{eq:uniform-real-positive}
	\sigma_i(W(\ba))>3\quad\text{and}\quad\sigma_i(T(\ba))>7\quad(i=1,2),
\end{equation}
and Theorem~\ref{thm:realpacking} \textup{(2)} with $r=1$ shows that
\begin{equation}\label{eq:uniform-real-packing}
W(\ba)\in\Q \quad\Longleftrightarrow\quad\ba\in\Z^n.
\end{equation}

Let $a,e$ be the fixed integers from \eqref{eq:real-ae}. Introduce new
unknowns $U,V,J$ and define
\begin{align*}
Y(\bX,U)&=T(\bX)R(\bX)U,\\
Q(\bX,U,V)&=T(\bX)+Y(\bX,U)^2V,\\
A_1(\bX,U)&=1+eY(\bX,U)^2,\\
A_2(\bX,U,V)&=1+eY(\bX,U)^2Q(\bX,U,V)^2.
\end{align*}
Set
\[
\Phi_{K,n,R}(\bX,U,V,J)=F\left(4A_1(\bX,U),A_2(\bX,U,V),Y(\bX,U),D(V),J\right).
\]
Because $e$ is even, the first two arguments of $F$ are distinct after every
evaluation over $\sO_K$. Also, for every $\alpha\in\sO_K$,
\begin{equation}\label{eq:uniform4square}
4\alpha\in\Sq{\sO_K}\quad\Longleftrightarrow\quad\alpha\in\Sq{\sO_K}.
\end{equation}

Suppose first that $\ba\in\Z^n$ and $R(\ba)\neq0$. Set
\[
w=W(\ba)\quad\text{and}\quad t=T(\ba)=2w+1.
\]
Then $w\in\Z$ and $w>3$ by \eqref{eq:uniform-real-positive}, so $t\in\Z_{>0}$ is odd. By Lemma~\ref{lem:pell-div}, choose $r_0\in\Z_{>0}$ such that
\[
tR(\ba)\mid q_{r_0}\text{ in }\Z,
\]
where $p_{m},q_{m}\in\Z_{\geq0}$ ($m\geq0$) are defined as in \S~\ref{sec:real}. Set
\[
y=q_{r_0}\quad\text{and}\quad u=\frac{q_{r_0}}{tR(\ba)}\in\Z.
\]
Then $y=Y(\ba,u)$. Applying Corollary~\ref{cor:reusex} with sign $+1$, we conclude that there are a positive odd integer $k$ and $v\in\Z$ such that
\[
q:=Q(\ba,u,v)=t+y^2v=\frac{q_{r_0k}}{q_{r_0}}\quad\text{and}\quad y\mid D(v)\text{ in }\Z.
\]
Consequently,
\[
A_1(\ba,u)=1+eq_{r_0}^2=p_{r_0}^2\quad\text{and}\quad A_2(\ba,u,v)=1+eq_{r_0}^2q^2=p_{r_0k}^2.
\]
Lemma~\ref{lem:F} for $\Z$ supplies $j\in\Z$ for which
$\Phi_{K,n,R}(\ba,u,v,j)=0$.

Conversely, suppose that
$\Phi_{K,n,R}(\ba,u,v,j)=0$ for $(\ba,u,v,j)\in\sO_K^{n+3}$, and set
\begin{align*}
t&=T(\ba),& y&=Y(\ba,u),& q&=Q(\ba,u,v),\\
a_1&=A_1(\ba,u),& a_2&=A_2(\ba,u,v).
\end{align*}
Lemma~\ref{lem:F} for $\sO_K$, \eqref{eq:Dnonzero}, and \eqref{eq:uniform4square} imply
\[
a_1,a_2\in\Sq{\sO_K}\quad\text{and}\quad y\mid D(v)\text{ in }\sO_K.
\]
Note that $y\neq0$. Since $a_1,a_2\in\Sq{\sO_K}$, Theorem~\ref{thm:Denef} yields
\[
y^2,(yq)^2\in\Z_{\geq0}.
\]
Since $y=tR(\ba)u$, one has $t\mid y$ in $\sO_K$. Set $\nu=y^2\in\Z_{>0}$ and write $t_i=\sigma_i(t)$. By \eqref{eq:uniform-real-positive}, we obtain $t_1,t_2>7$. Since $t\mid y$ in $\sO_K$, write $y=tc$ with $c\in\sO_K\setminus\{0\}$. Moreover,
\[
\Norm_{K/\Q}(y)^2=\Norm_{K/\Q}(y^2)=\Norm_{K/\Q}(\nu)=\nu^2,
\]
so $|\Norm_{K/\Q}(y)|=\nu$. Hence
\[
\nu=|\Norm_{K/\Q}(y)|=|\Norm_{K/\Q}(t)|\,|\Norm_{K/\Q}(c)|\geq |\Norm_{K/\Q}(t)|.
\]
Since $t_1,t_2>7$, we obtain
\[
\nu\geq\Norm_{K/\Q}(t)>49.
\]
Moreover,
\[
\Norm_{K/\Q}(t)-\Tr_{K/\Q}(t)=(t_1-1)(t_2-1)-1>35,
\]
and then
\[
\nu\geq\Norm_{K/\Q}(t)>\Tr_{K/\Q}(t)=t_1+t_2>14.
\]
Therefore $\nu\nmid\Tr_{K/\Q}(t)$ in $\Z$, while
\[
|t_1-t_2|<t_1+t_2=\Tr_{K/\Q}(t)<\nu<\nu\sqrt{d}\leq\nu\sqrt{\Delta_K}.
\]
Since $y^2,(yq)^2\in\Z_{\geq0}$ and $y\neq0$, we also have $q^2=(yq)^2/y^2\in\Q$. Proposition~\ref{prop:real-rationality}, applied with $q=t+\nu v$, yields $t\in\Z$; hence $W(\ba)=(t-1)/2\in\Q$. Thus, $\ba\in\Z^n$ by \eqref{eq:uniform-real-packing}. Finally, $y=tR(\ba)u\neq0$ implies $R(\ba)\neq0$.

\smallskip

The construction is effective. In the imaginary case, every polynomial above is given explicitly in terms of $n$ and $R$. In the real case, one first finds a nontrivial solution of $a_0^2-db_0^2=1$ by the continued-fraction algorithm and then performs the displayed polynomial constructions. Hence, $\Phi_{K,n,R}$ is effectively constructible from $(K,n,R)$.
\end{proof}

\begin{proof}[Proof of Theorem~\ref{thm:plus3}]
Set
\[
R=\prod_{j=1}^{t}R_j,
\]
with $R=1$ if $t=0$. For $(\ba,\bb)\in\sO_K^{r+m}$, we have
\[
R(\ba,\bb)\neq0\quad\Longleftrightarrow\quad R_j(\ba,\bb)\neq0\quad(1\leq j\leq t).
\]

If $s=0$, set $P=0$. If $s>0$ and $K$ is real, set
\[
P=\sum_{i=1}^{s}P_i^2.
\]
If $s>0$ and $K=\Q(\sqrt{-d})$ is imaginary with $d$ a positive squarefree integer, set $E_1=P_1$ and recursively
\[
E_{i+1}=E_i^2+(d+1)P_{i+1}^2\quad(1\leq i<s),
\]
and set $P=E_s$. In the imaginary case, Lemma~\ref{lem:anisotropic} shows inductively that over $\sO_K$, $P=0$ is equivalent to $P_1=\cdots=P_s=0$; the corresponding equivalence in the real case is immediate.

Apply Theorem~\ref{thm:uniform3} to the $r+m$ unknowns $(\bX,\bY)$ and the polynomial $R$. Let
\[
I(\bX,\bY,U,V,J)\in\Z[\bX,\bY,U,V,J]
\]
be the resulting integrality polynomial. Thus, for all $(\ba,\bb)\in\sO_K^{r+m}$,
\[
\exists u,v,j\in\sO_K\colon\ I(\ba,\bb,u,v,j)=0\quad\Longleftrightarrow\quad(\ba,\bb)\in\Z^{r+m}\text{ and }R(\ba,\bb)\neq0.
\]

If $K$ is real, define $H=P^2+I^2$; if $K=\Q(\sqrt{-d})$ is imaginary, define $H=P^2+(d+1)I^2$. In the real case, we have $H=0$ is equivalent to $P=I=0$ over $\sO_K$, while in the imaginary case the same conclusion follows from Lemma~\ref{lem:anisotropic}. This gives the asserted equivalence with the $3$ additional quantified unknowns $U,V,J$. Effectivity follows from Theorem~\ref{thm:uniform3} and the constructions above.
\end{proof}

\begin{remark}\label{rmk:OKcoeff}
Theorems~\ref{thm:uniform3} and~\ref{thm:plus3} remain valid if the input polynomial(s) are allowed to have coefficients in $\sO_K$. More precisely, in Theorem~\ref{thm:uniform3} one may take
\[
R(\bX)\in\sO_K[\bX],
\]
and in Theorem~\ref{thm:plus3} one may take
\[
P_1,\ldots,P_s,R_1,\ldots,R_t\in\sO_K[\bX,\bY].
\]
In both cases, the resulting polynomial can still be constructed effectively from $K$ and the input polynomial(s); in fact, it can still be taken to have coefficients in $\Z$.

Indeed, let $\sigma$ denote the nontrivial automorphism of $K/\Q$. For $G\in\sO_K[\bZ]$, where $\bZ$ is any finite tuple of variables, define $N_G:=G\cdot G^\sigma$, where $\sigma$ acts coefficientwise. Then $N_G\in\Z[\bZ]$. Moreover, for every $\bz\in\Z^{\#\bZ}$, one has $N_G(\bz)=G(\bz)G(\bz)^\sigma=\Norm_{K/\Q}\bigl(G(\bz)\bigr)$, hence
\[
G(\bz)=0\quad\Longleftrightarrow\quad N_G(\bz)=0\qquad\text{and}\qquad G(\bz)\neq0\quad\Longleftrightarrow\quad N_G(\bz)\neq0.
\]

Thus, for Theorem~\ref{thm:uniform3} (resp.~Theorem~\ref{thm:plus3}), one applies the original theorem to $N_R$ (resp.~$N_{P_1},\ldots,N_{P_s},N_{R_1},\ldots,N_{R_t}$). Since coefficientwise conjugation and multiplication in $K$ are effective, the resulting polynomial is effectively constructible from $K$ and the original input polynomial(s).
\end{remark}

\subsection{Recursively enumerable relations}\label{sec:re}
\begin{theorem}[$13$-unknown representation]\label{thm:re13}
Let $K$ be a quadratic number field, let $r\in\Z_{>0}$, and let $\sD\subseteq\Z^r$ be r.e. There is a polynomial
\[
H_{\sD}(X_1,\ldots,X_r,Y_1,\ldots,Y_{13})\in\Z[\bX,\bY]
\]
such that, for every $\ba\in\sO_K^r$,
\[
\ba\in\sD\quad\Longleftrightarrow\quad\exists\bb\in\sO_K^{13}\colon\ H_{\sD}(\ba,\bb)=0.
\]
The number $13$ is independent of $r$ and of $\sD$.
\end{theorem}
\begin{proof}
We first encode $\Z^r$ in $\Z_{\geq0}$ using polynomials. Define
\[
G_0(T)=2T^2-T\quad\text{and}\quad C(U,V)=(U+V)^2+U.
\]
It is easy to verify $G_0(\Z)\subseteq\Z_{\geq0}$ and $C(\Z_{\geq0}^2)\subseteq\Z_{\geq0}$. Since
\[
G_0(a)-G_0(b)=(a-b)(2a+2b-1),
\]
the map $G_0:\Z\to\Z_{\geq0}$ is injective. The map $C:\Z_{\geq0}^2\to\Z_{\geq0}$ is also injective: if $s=u+v$ with $u,v\in\Z_{\geq0}$, then
$s^2\leq C(u,v)<(s+1)^2$, so $C(u,v)$ determines $s$, then $u$, and then $v$.
Define $J_1=G_0$ and recursively
\[
J_{r+1}(X_1,\ldots,X_{r+1})
=C\left(J_r(X_1,\ldots,X_r),G_0(X_{r+1})\right).
\]
Then $J_r:\Z^r\to\Z_{\geq0}$ is an injective polynomial map.

Set $\sA=J_r(\sD)$. The set $\sA$ is r.e.\ because $\sD$ is r.e. Then Theorem~\ref{thm:source} gives a polynomial
\[
Q_{\sA}(T,Z_1,\ldots,Z_{10})\in\Z[T,Z_1,\ldots,Z_{10}]\]
such that for every $a\in\Z_{\geq0}$,
\[
a\in\sA\quad\Longleftrightarrow\quad
\exists z_1,\ldots,z_9\in\Z\ \exists z_{10}\in\Z\setminus\{0\}:
Q_{\sA}(a,z_1,\ldots,z_{10})=0.
\]
Hence, for $\ba\in\Z^r$,
\[
\ba\in\sD\quad\Longleftrightarrow\quad
\exists\bz\in\Z^{10}:
Q_{\sA}(J_r(\ba),\bz)=0,\quad z_{10}\neq0.
\]
Apply Theorem~\ref{thm:plus3} with $(m,s,t)=(10,1,1)$ and
\[
P_1(\bX,\bZ)=Q_{\sA}(J_r(\bX),\bZ),\quad R_1(\bX,\bZ)=Z_{10}.
\]
The resulting representation over $\sO_K$ has $10+3=13$ unknowns.
\end{proof}

\begin{proof}[Proof of Theorem~\ref{thm:bd13}]
Fix a nonrecursive r.e.\ set $\sA\subseteq\Z_{\geq0}$. Applying Theorem~\ref{thm:re13} with $r=1$ gives, for every quadratic number field $K$, a polynomial
\[
H_{K,\sA}(T,Y_1,\ldots,Y_{13})\in\Z[T,Y_1,\ldots,Y_{13}]
\]
such that, for every $t\in\sO_K$,
\[
t\in \sA\quad\Longleftrightarrow\quad
\exists\by\in\sO_K^{13}: H_{K,\sA}(t,\by)=0.
\]

Since $\sA$ is fixed, the polynomial $Q_{\sA}$ given by Theorem~\ref{thm:source} has fixed degree. In the proof of Theorem~\ref{thm:re13}, we have $r=1$, so the encoding polynomial is $J_1=G_0$, which has degree $2$. We also apply Theorem~\ref{thm:plus3} with $(m,s,t)=(10,1,1)$ and $R_1=Z_{10}$, which has degree $1$. Thus, the numbers of variables and the exponents appearing in the transfer construction are fixed. It remains to check that the quadratic field $K$ affects only the coefficients, not the degrees. In the imaginary quadratic case, the squarefree integer $d$ appears only in coefficients such as $d+1$. Thus, it does not affect the total degree. In the real quadratic case, $d$ and the chosen solution $(a_0,b_0)$ of $a_0^2-db_0^2=1$ enter only through the coefficients $a$ and $e$. Again, they do not affect the exponents in the defining polynomials. Therefore, the degrees arising from the construction depend only on $\deg(Q_{\sA})$ and on whether $K$ is real or imaginary. They do not depend on the particular quadratic field $K$. Taking the maximum of the two resulting degree bounds gives an integer $D_0\geq1$, independent of $K$, such that $H_{K,\sA}(t,Y_1,\ldots,Y_{13})$ can always be chosen to have total degree at most $D_0$.

If the decision algorithm excluded by Theorem~\ref{thm:bd13} exists for a fixed quadratic $K$, then on input $t\in\Z$, one could apply it to the polynomial $H_{K,\sA}(t,Y_1,\ldots,Y_{13})$ and decide membership in $\sA$, contrary to the assumption that $\sA$ is nonrecursive.
\end{proof}

\subsection{Universal pairs}
Theorem~\ref{thm:bd13} gives a uniform degree bound, but not an explicit one. We now obtain an explicit bound at the cost of one additional quantified unknown.

Let $R$ be an integral domain of characteristic zero, and view $\Z\subseteq R$. A pair $(\nu,\Delta)\in\Z_{>0}^2$ is called \emph{universal} over $R$ if, for every r.e.\ set $\sA\subseteq\Z_{\geq0}$, there exists a polynomial
\[
P_{\sA}(T,Y_1,\ldots,Y_\nu)\in\Z[T,Y_1,\ldots,Y_\nu]
\]
of total degree at most $\Delta$ such that, for every $t\in\Z_{\geq0}$,
\[
t\in\sA\quad\Longleftrightarrow\quad\exists\bb\in R^\nu\colon\ P_{\sA}(t,\bb)=0.
\]
Thus, the parameter $t$ is already restricted to $\Z_{\geq0}$ in the definition; only the $\nu$ quantified unknowns range over $R$.

Bayer--David--Hassler--Matiyasevich--Schleicher \cite{BayerEtAl25}, together with the formal verification of Bayer--David \cite{BayerDavid25}, proved that $(11,\Delta_{\Z})$ is universal over $\Z$, where
\[
\begin{aligned}
	\Delta_{\Z}&=1681043235226619916301182624511918527834137733707408448335539840\\
	&\approx1.68105\cdot10^{63}.
\end{aligned}
\]
Bayer et al. \cite{BayerEtAl25} also state a universal pair $(11,\Delta_{\Z}')$ over $\Z$ with $\Delta_{\Z}'\approx9.50818\cdot10^{53}\ll\Delta_{\Z}$. However, the correctness of this smaller pair depends on a statement of Jones, which is not proved. See \cite[\S~0.3]{BayerEtAl25} for details.

\begin{theorem}[Universal pair for quadratic integer rings]\label{thm:universalpairOK}
For every quadratic number field $K$, the pair $(14,2\Delta_{\Z})$ is universal over $\sO_K$. Explicitly,
\[
\begin{aligned}
	2\Delta_{\Z}&=3362086470453239832602365249023837055668275467414816896671079680\\
	&\approx3.36209\cdot10^{63}.
\end{aligned}
\]
\end{theorem}
\begin{proof}
Let $\sA\subseteq\Z_{\geq0}$ be r.e., and let
$P_{\sA}(T,Y_1,\ldots,Y_{11})$ be an integer-representation polynomial of degree at most $\Delta_{\Z}$. Apply Theorem~\ref{thm:uniform3} with $n=11$ and $R=1$ to the $11$ quantified unknowns $Y_1,\ldots,Y_{11}$; the parameter $T$ is not included because it is already restricted to $\Z_{\geq0}$. Let
\[
I(Y_1,\ldots,Y_{11},U,V,J)=\Phi_{K,11,1}(Y_1,\ldots,Y_{11},U,V,J)
\]
be the resulting polynomial. Then for every $\ba\in\sO_K^{11}$,
\[
\exists u,v,j\in\sO_K\colon\ I(\ba,u,v,j)=0\quad\Longleftrightarrow\quad\ba\in\Z^{11}.
\]
Define
\[
H_{\sA}=
\begin{cases}
P_{\sA}^2+I^2, & K\text{ real},\\[2mm]
P_{\sA}^2+(d+1)I^2, & K=\Q(\sqrt{-d})\text{ with }d\text{ positive and squarefree}.
\end{cases}
\]
For $t\in\Z_{\geq0}$, the equation $H_{\sA}(t,\cdot)=0$ over $\sO_K$ is equivalent to $P_{\sA}(t,\cdot)=0$ having a solution in $\Z^{11}$. 

It remains to verify that $\deg(H_{\sA})\leq 2\Delta_\Z$. In the imaginary construction in the proof of Theorem~\ref{thm:uniform3}, with $n=11$ and $R=1$, the degrees of $L_0,L_1,L_2,H,S,Y,B,C$ are respectively
\[
11,\ 121,\ 132,\ 506,\ 1,\ 507,\ 1014,\ 3044,
\]
and direct substitution into $F$ gives $\deg I\leq6092$. In the real construction, the degrees of $M,W,T,Y,Q,A_1,A_2$ are respectively
\[
22,\ 264,\ 264,\ 265,\ 531,\ 530,\ 1592,
\]
and $\deg I\leq4244$. Therefore
\[
\deg H_{\sA}\leq\max\{2\Delta_{\Z},2\deg I\}=2\Delta_{\Z},
\]
which proves the conclusion.
\end{proof}

The universal pair has the following undecidable consequence:
\begin{corollary}\label{cor:explicit-bounded}
For every quadratic number field $K$, there is no algorithm that, given a polynomial $R(\bX)\in\Z[X_1,\ldots,X_{14}]$ of total degree at most $2\Delta_{\Z}$, decides whether $R=0$ has a solution in $\sO_K^{14}$.
\end{corollary}
\begin{proof}
Choose a nonrecursive r.e.\ set $\sA\subseteq\Z_{\geq0}$. By Theorem~\ref{thm:universalpairOK}, there exists a polynomial
\[
H_\sA(T,Y_1,\ldots,Y_{14})\in\Z[T,Y_1,\ldots,Y_{14}]
\]
of total degree at most $2\Delta_{\Z}$ such that, for every $t\in\Z_{\geq0}$,
\[
t\in\sA\quad\Longleftrightarrow\quad
\exists\by\in\sO_K^{14}:H_\sA(t,\by)=0.
\]
If the asserted decision algorithm exists, then for an input $t\in\Z_{\geq0}$, it would decide membership in $\sA$, contradicting the nonrecursiveness of $\sA$.
\end{proof}
Thus, Theorem~\ref{thm:bd13} gives $13$ unknowns with a non-explicit uniform degree bound $D_0$, while Corollary~\ref{cor:explicit-bounded} gives the explicit degree bound $2\Delta_{\Z}$ with $14$ unknowns.

\section{Lucas applications}\label{sec:pure-apps}
The fourth-order congruence used in the $3$-unknown construction is the first nontrivial truncation of a classical Lucas multiplication formula. We present that polynomial explicitly for the norm-one family, then use its first nonzero terms to study the exact $\ell$-adic valuation of the deviation $u_{nk}/u_{n}-k$. Standard valuations of Lucas terms are classical; see, for example, \cite{Sanna16,BW23}. The final results concern instead the local distribution of these deviations, which is related to, but distinct from, the $p$-adic quotient-set problem studied in \cite{AntonyBarman23,AntonyBarman25}.

\subsection{Lucas expansions and deviations}\label{sec:Lucas}
Fix $a\in\Z$. Define $(u_n)_{n\geq0},(v_n)_{n\geq0}$ by
\[
u_0=0,\quad u_1=1,\quad u_{n+2}=au_{n+1}-u_n,
\qquad
v_0=2,\quad v_1=a,\quad v_{n+2}=av_{n+1}-v_n.
\]
Write $Z^2-aZ+1=(Z-\alpha)(Z-\beta)$, where $\alpha,\beta\in\overline{\Q}$. Then for $n\geq0$, we have
\[
u_n=\sum_{i=0}^{n-1}\alpha^{n-1-i}\beta^i,\quad v_n=\alpha^n+\beta^n,
\]
and
$u_n=(\alpha^n-\beta^n)/(\alpha-\beta)$ if $\delta:=a^2-4=(\alpha-\beta)^2\neq0$. Note that $v_n^2-\delta u_n^2=4$ for $n\geq0$.

The next proposition is a reformulation of the classical Lucas multiplication formula; see \cite[\S~2.2]{BW23}.
\begin{proposition}[All-orders expansion]\label{prop:lucas-all}
Let $k\geq1$ be an odd integer, and set $\theta=(k-1)/2\in\Z_{\geq0}$. For an integer $0\leq r\leq\theta$, define
\[
c_r(k)=\frac{k\prod_{j=1}^{r}\left(k^2-(2j-1)^2\right)}{4^r(2r+1)!},
\]
where $c_0(k)=k$; then $c_r(k)\in\Z_{>0}$. For every $n\in\Z_{>0}$ with $u_n\neq0$, one has
\[
\frac{u_{nk}}{u_n}=\sum_{r=0}^{\theta}c_r(k)\delta^r u_n^{2r};
\]
in particular, for every $h\in\Z_{>0}$, one has
\[
\frac{u_{nk}}{u_n}\equiv\sum_{r=0}^{\min\{h-1,\theta\}}c_r(k)\delta^r u_n^{2r}\pmod{u_n^{2h}}.
\]
\end{proposition}
\begin{proof}
By the classical Lucas multiplication polynomial \cite[(2.35)]{BW23}, for every $n\ge1$ with $u_n\neq0$, one has
\[
\frac{u_{nk}}{u_n}=\Lambda_k(v_n),
\]
where
\[
\Lambda_k(Z)=\sum_{r=0}^{\theta}(-1)^{\theta-r}\binom{k+r-\theta-1}{\theta-r}Z^{2r}\in\Z[Z]
\]
is an even integer-coefficient polynomial depending only on $k$. Write
\[
\Lambda_k(Z)=\Phi_k(Z^2-4),
\]
where $\Phi_k(Z)\in\Z[Z]$.

Using the standard identities
\[
\Lambda_k(2\cosh t) = \frac{\sinh(kt)}{\sinh t}\quad\text{and}\quad\frac{\sinh(kt)}{\sinh t}=\sum_{r=0}^{\theta}4^rc_r(k)\sinh^{2r}t
\]
and substituting $W=Z^2-4$ alongside $(Z,W)=(2\cosh t,4\sinh^2t)$ yield
\[
\Phi_k(W)=\sum_{r=0}^{\theta}c_r(k)W^r.
\]
So $c_r(k)\in\Z_{>0}=\Z\cap\Q_{>0}$ for $0\leq r\leq\theta$. Since $v_n^2-4=\delta u_n^2$, substitution yields the exact formula, and truncation shows the congruence.
\end{proof}

The fourth-order formula is the first nonconstant truncation of this multiplication polynomial.
\begin{corollary}[Fourth-order expansion]\label{cor:lucas4}
Let $n\in\Z_{>0}$ with $u_n\neq0$, and let $k\in\Z_{>0}$ be odd. Then $u_n\mid u_{nk}$ and
\[
\frac{u_{nk}}{u_n}\equiv k+\frac{\delta k(k^2-1)}{24}u_n^2\pmod{u_n^4}.
\]
Moreover, there are $\Psi_k, \Psi_k^*\in\Z[W]$ such that
\[
v_{nk}=v_n\Psi_k(\delta u_n^2)\quad\text{and}\quad\Psi_k(W)=1+\frac{k^2-1}{8}W+W^2\Psi_k^*(W).
\]
\end{corollary}
\begin{proof}
The first two assertions follow from Proposition~\ref{prop:lucas-all} and
\[
c_1(k)=\frac{k(k^2-1)}{24}.
\]
The statement for $(v_{nk},v_n)$ can be proved as in the proof of Proposition~\ref{prop:lucas-all}. We omit the details.
\end{proof}

For a prime $\ell$, let $v_\ell$ denote the usual $\ell$-adic valuation on $\Q$, with $v_\ell(0)=\infty$. The next corollary measures the exact valuation of the error after subtracting the linear term $k$.
\begin{corollary}[Exact deviation valuation]\label{cor:deviation-vp}
	Let $n,k\in\Z_{>0}$ with $u_n\neq0$ and $k$ odd, and let $\ell\geq3$ be a prime. Suppose that $\ell\nmid\delta$ and $\ell\mid u_n$. Set $r=v_\ell(u_n)\geq1$. Define $h_\ell=1$ if $\ell=3$, and $h_\ell=0$ otherwise.
	\begin{enumerate}[label=\textup{(\arabic*)}]
		\item Set $\tau=v_\ell(k(k^2-1))$. If $\tau<2r+h_\ell$, then
		\[
		v_\ell\left(\frac{u_{nk}}{u_n}-k\right)=2r+\tau-h_\ell\quad\text{and}\quad v_\ell(u_{nk}-ku_n)=3r+\tau-h_\ell.
		\]
		\item Set $\tau'=v_\ell(k^2-1)$. If $\tau'<2r$, then $v_\ell(v_{nk}-v_n)=2r+\tau'$.
	\end{enumerate}
\end{corollary}
\begin{proof}
	\textup{(1)}. By Corollary~\ref{cor:lucas4},
	\[
	\frac{u_{nk}}{u_n}-k=\frac{\delta k(k^2-1)}{24}u_n^2+u_n^4c
	\]
	for some $c\in\Z$. The two terms in the right-hand side have $\ell$-adic valuations $2r+\tau-h_\ell$ and at least $4r$, respectively. The hypothesis $\tau<2r+h_\ell$ shows $4r>2r+\tau-h_\ell$. Hence, the first equality holds, and the second equality follows.
	
	\textup{(2)}. Likewise,
	\[
	v_{nk}-v_n=v_n\frac{\delta(k^2-1)}8u_n^2+v_nu_n^4c'
	\]
	for some $c'\in\Z$. Since $v_n^2=4+\delta u_n^2\equiv4\pmod{\ell}$ and $\ell\geq3$, one has $\ell\nmid v_n$. The two terms in the right-hand side have valuations $2r+\tau'$ and at least $4r$, respectively. Then the conclusion follows.
\end{proof}

\subsection{Wieferich deviations}
Let $\ell\geq3$ be a rational prime with $\ell\nmid\delta$, and suppose that $\ell$ divides some nonzero term of $(u_n)$. Let $\rho_\ell\ge1$ be the smallest index such that $u_{\rho_\ell}\neq0$ and $\ell\mid u_{\rho_\ell}$. Define $r_\ell=v_\ell(u_{\rho_\ell})$. We call $\ell$ a \emph{Lucas--Wieferich prime} for the sequence $(u_n)_{n\geq1}$ if $r_\ell\geq2$. Set $h_\ell=1$ if $\ell=3$, and $h_\ell=0$ otherwise. The following corollary shows that $r_\ell$ is recovered from the deviation of the multiplication formula from its linear term:
\begin{corollary}[Lucas--Wieferich deviation]\label{cor:wieferich-deviation}
	With the preceding notation, we have
	\[
	v_\ell(u_{\ell\rho_\ell}-\ell u_{\rho_\ell})=3r_\ell+1-h_\ell\quad\text{and}\quad v_\ell(v_{\ell\rho_\ell}-v_{\rho_\ell})=2r_\ell.
	\]
	Consequently,
	\[
	\begin{aligned}
		\ell^2\mid u_{\rho_\ell}\quad&\Longleftrightarrow\quad\ell^{7-h_\ell}\mid u_{\ell\rho_\ell}-\ell u_{\rho_\ell}\quad\Longleftrightarrow\quad\ell^4\mid v_{\ell\rho_\ell}-v_{\rho_\ell}\\
		&\Longleftrightarrow\quad\ell^{5-h_\ell}\mid u_{\ell\rho_\ell}-\ell u_{\rho_\ell}\quad\Longleftrightarrow\quad\ell^3\mid v_{\ell\rho_\ell}-v_{\rho_\ell},
	\end{aligned}
	\]
	and
	\[
	r_\ell=
	\frac{v_\ell(u_{\ell\rho_\ell}-\ell u_{\rho_\ell})+h_\ell-1}{3}=\frac{v_\ell(v_{\ell\rho_\ell}-v_{\rho_\ell})}{2}.
	\]
\end{corollary}
\begin{proof}
	Applying Corollary~\ref{cor:deviation-vp} with $(n,k)=(\rho_\ell,\ell)$ yields $\tau=1$ and $\tau'=0$. Since $1\leq r_\ell<2r_\ell+h_\ell$ and $0<2r_\ell$, the valuation conditions are satisfied, yielding the conclusion.
\end{proof}

\medskip

We now consider Fibonacci numbers. Let $(f_n)_{n\geq0}$ be the Fibonacci sequence, and $(\lambda_n)_{n\geq0}$ its Lucas companion:
\[
f_0=0,\quad f_1=1,\quad f_{n+2}=f_{n+1}+f_n,\quad\lambda_0=2,\quad\lambda_1=1,\quad\lambda_{n+2}=\lambda_{n+1}+\lambda_n.
\]
Using the notation above with $a=3$, one has
\[
u_n=f_{2n},\quad v_n=\lambda_{2n},\quad\delta=5.
\]

For a prime $\ell$, let $\pi(\ell)$ be the period of the sequence $(f_n\pmod{\ell})_{n\ge0}$ (it always exists), which is called the \emph{Pisano period modulo $\ell$}. A rational prime $\ell$ is called a \emph{Wall--Sun--Sun prime} if $\pi(\ell^2)=\pi(\ell)$; see \cite{Wall60,SS92,HarringtonJones23}. Let $\alpha_\ell$ be the \emph{Fibonacci rank of apparition of $\ell$}, i.e., the smallest index $\alpha_\ell\geq1$ such that $\ell\mid f_{\alpha_\ell}$. For an odd prime $\ell\geq3$, the classical properties of Fibonacci numbers imply
\[
\pi(\ell^2)=\pi(\ell)\quad\Longleftrightarrow\quad \ell^2\mid f_{\alpha_\ell};
\]
see, for example, \cite{Wall60} and \cite[Theorem~2]{Vinson63}.

The following corollary expresses the Wall--Sun--Sun condition through the exact linearization deviations:
\begin{corollary}[Fibonacci deviation]\label{cor:fib-deviation}
	Suppose that $\ell$ is an odd prime with $\ell\neq5$. Set
	\[
	\rho_\ell=\frac{\alpha_\ell}{\gcd(\alpha_\ell,2)}\quad\text{and}\quad r_\ell=v_\ell(f_{\alpha_\ell}).
	\]
	Then
	\[
	v_\ell\left(f_{2\ell\rho_\ell}-\ell f_{2\rho_\ell}\right)=3r_\ell+1-h_\ell\quad\text{and}\quad v_\ell\left(\lambda_{2\ell\rho_\ell}-\lambda_{2\rho_\ell}\right)=2r_\ell.
	\]
	Therefore,
	\[
	\begin{aligned}
		\ell\text{ is a Wall--Sun--Sun prime}\quad&\Longleftrightarrow\quad\ell^{7-h_\ell}\mid f_{2\ell\rho_\ell}-\ell f_{2\rho_\ell}\quad\Longleftrightarrow\quad \ell^4\mid\lambda_{2\ell\rho_\ell}-\lambda_{2\rho_\ell}\\
		&\Longleftrightarrow\quad\ell^{5-h_\ell}\mid f_{2\ell\rho_\ell}-\ell f_{2\rho_\ell}\quad\Longleftrightarrow\quad \ell^3\mid\lambda_{2\ell\rho_\ell}-\lambda_{2\rho_\ell},
	\end{aligned}
	\]
	and
	\[
	r_\ell=\frac{v_\ell(f_{2\ell\rho_\ell}-\ell f_{2\rho_\ell})+h_\ell-1}{3}=\frac{v_\ell(\lambda_{2\ell\rho_\ell}-\lambda_{2\rho_\ell})}{2}.
	\]
\end{corollary}
\begin{proof}
	The least index $n\ge1$ with $\ell\mid f_{2n}$ is $\rho_\ell$. If $\alpha_\ell$ is even, then $2\rho_\ell=\alpha_\ell$. If $\alpha_\ell$ is odd, then the classical identities $f_{2m}=f_{m}\lambda_{m}$ and $\lambda_m^2-5f_m^2=4(-1)^m$ give $v_\ell(f_{2\rho_\ell})=v_\ell(f_{\alpha_\ell})=r_\ell$. In either case, we have $v_\ell(f_{2\rho_\ell})=r_\ell$. Applying Corollary~\ref{cor:wieferich-deviation} to $u_n=f_{2n}$ implies the conclusion.
\end{proof}

\subsection{Local surjectivity}
The expansions in \S~\ref{sec:Lucas} also have a local consequence: the normalized second-order correction can be prescribed modulo every divisor of the base Lucas term.

\begin{theorem}[Local surjectivity]\label{thm:lucas-surj}
Fix $n\in\Z_{>0}$ with $u_n\neq0$. Let $s\in\Z\setminus\{0\}$ satisfy $s\mid u_n$, and let $t\in\Z$ be odd. Then, for every residue class $c\pmod{s}$, there are infinitely many positive odd integers $k$ such that
\[
k\equiv t\pmod{u_n^2}\quad\text{and}\quad\frac{u_{nk}/u_n-t}{u_n^2}\equiv c\pmod{s}.
\]
\end{theorem}
\begin{proof}
For a positive odd integer $k=t+u_n^2r$ ($r\in\Z$), Corollary~\ref{cor:lucas4} gives
\[
\frac{u_{nk}/u_n-t}{u_n^2}=r+\delta\frac{k(k^2-1)}{24}+u_n^2b_k
\]
for some $b_k\in\Z$.

We claim that
\[
E:=\delta\frac{k(k^2-1)}{24}-\delta\frac{t(t^2-1)}{24}\equiv0\pmod{s}.
\]
Observe that
\[
k(k^2-1)-t(t^2-1)=(k-t)(k^2+kt+t^2-1)
\]
and $s^2\mid (k-t)$. Suppose $v_p(s)=e$ with $p$ a prime and $e\in\Z_{>0}$. If $p\ge3$, then
\[v_p(E)\geq v_p(k-t)-v_p(24)\geq v_p(s^2)-1=2e-1\geq e.
\]
Suppose now $p=2$. If $a$ is even, then $v_2(\delta)=v_2(a^2-4)\geq2$, so
\[
v_2(E)\geq v_2(k-t)+v_2(\delta)-v_2(24)\geq v_2(s^2)+2-3=2e-1\geq e.
\]
Suppose now that $a$ is odd. As $s,u_n$ are even, the recurrence for $(u_j)_{j\ge0}$ modulo $2$ yields $3\mid n$. Since $u_3=a^2-1\equiv0\pmod{8}$ and $u_3\mid u_n$ by \cite[(2.35)]{BW23}, we conclude that
\[
r_2:=v_2(u_n)\geq\max\{e,3\}.\]
Hence
\[
v_2(E)\geq v_2(k-t)-v_2(24)\geq2r_2-3\geq r_2\geq e.
\]
From the claimed congruence, we deduce that
\[
\frac{u_{nk}/u_n-t}{u_n^2}\equiv r+\delta\frac{t(t^2-1)}{24}\pmod{s}.
\]

Taking $r$ in the residue class
\[
r\equiv c-\delta\frac{t(t^2-1)}{24}\pmod{s}
\]
implies
\[
\frac{u_{nk}/u_n-t}{u_n^2}\equiv c\pmod{s}.
\]
The argument used in the proof of Proposition~\ref{prop:pell-star}\textup{(2)}, with $q_n$ replaced by $u_n$, shows that this residue class contains infinitely many $r\in\Z$ for which
\[
k=t+u_n^2r
\]
is a positive odd integer. Therefore, there are infinitely many $k$ satisfying both required congruences.
\end{proof}

\begin{corollary}[$\ell$-adic density]\label{cor:padic-density}
Let $\ell\geq3$ be a prime with $\ell\nmid\delta$. Let $n_0\in\Z_{>0}$ satisfy
\[
u_{n_0}\neq0\quad\text{and}\quad r_0:=v_\ell(u_{n_0})\geq1.
\]
Define $N_j:=n_0\ell^j$ for $j\in\Z_{\geq0}$. Fix an odd integer $t$. Then, for every $M\in\Z_{>0}$, every residue class $c\pmod{\ell^M}$, and every integer $j\geq\max\{0,M-r_0\}$, one has
\[
\#\left\lbrace k\in\Z_{>0}\colon2\nmid k,\ k\equiv t\pmod{u_{N_j}^2},\ \frac{u_{N_jk}/u_{N_j}-t}{u_{N_j}^2}\equiv c\pmod{\ell^M}\right\rbrace=\infty.
\]
\end{corollary}
\begin{proof}
Note that $\delta\neq0$ because $\ell\nmid\delta$; hence $a\neq\pm2$. If $|a|\leq1$, then every nonzero term $u_m$ ($m\geq0$) is $\pm1$, contrary to $\ell\mid u_{n_0}$. We conclude that $|a|\geq3$. Thus, \cite[Corollary~1.6]{Sanna16} applies and yields $v_\ell(u_{N_j})=r_0+j$ for $j\geq0$. Now let $M\in\Z_{>0}$, let $c$ be a residue class modulo $\ell^M$, and fix an integer $j\geq\max\{0,M-r_0\}$. Then $v_\ell(u_{N_j})=r_0+j\geq M$, and hence $\ell^M\mid u_{N_j}$. Applying Theorem~\ref{thm:lucas-surj} with $(n,s)=(N_j,\ell^M)$ implies the assertion.
\end{proof}
Corollary~\ref{cor:padic-density} concerns normalized second-order corrections rather than ordinary quotient sets in $\Q_\ell$; compare \cite{AntonyBarman23,AntonyBarman25}.

\subsection*{AI Disclosure}
The author used AI tools (ChatGPT by OpenAI and Gemini by Google) to analyze Sun's paper \cite{Sun26}, assist with calculations, conduct literature searches and reference checks, and polish the language. The author independently verified and approved all mathematical proofs, conclusions, and the final text, who takes full responsibility for the paper's content and correctness.


\begin{thebibliography}{BDH$^+$25}

\bibitem[ABH$^+$26]{ABHS26}
Levent~Alp\"oge, Manjul~Bhargava, Wei~Ho, and Ari~Shnidman.
\newblock Rank stability in quadratic extensions and Hilbert's tenth problem for the ring of integers of a number field.
\newblock {\em Invent. Math.}, 243(3):1129--1139, 2026.

\bibitem[AB23]{AntonyBarman23}
Deepa~Antony and Rupam~Barman.
\newblock $p$-adic quotient sets: linear recurrence sequences.
\newblock {\em Bull. Aust. Math. Soc.}, 108(1):19--28, 2023.

\bibitem[AB25]{AntonyBarman25}
Deepa~Antony and Rupam~Barman.
\newblock $p$-adic quotient sets: linear recurrence sequences with reducible characteristic polynomials.
\newblock {\em Canad. Math. Bull.}, 68(1):177--186, 2025.

\bibitem[BD25]{BayerDavid25}
Jonas~Bayer and Marco~David.
\newblock A formal proof of complexity bounds on Diophantine equations.
\newblock In {\em 16th International Conference on Interactive Theorem Proving (ITP 2025)},
Leibniz International Proceedings in Informatics (LIPIcs), vol.~352, Art.~No.~3,
pp.~3:1--3:18, Schloss Dagstuhl--Leibniz-Zentrum f\"ur Informatik, 2025.

\bibitem[BDH$^+$25]{BayerEtAl25}
Jonas~Bayer, Marco~David, Malte~Hassler, Yuri~Matiyasevich, and Dierk~Schleicher.
\newblock Diophantine equations over $\Z$: universal bounds and parallel formalization.
\newblock arXiv:2506.20909v2, 2025.

\bibitem[BW23]{BW23}
Christian~J.-C.~Ballot and Hugh~C.~Williams.
\newblock {\em The Lucas Sequences: Theory and Applications}.
\newblock CMS/CAIMS Books in Math., vol.~8. Springer, Cham, 2023.

\bibitem[Cut80]{Cutland}
Nigel~Cutland.
\newblock {\em Computability: an introduction to recursive function theory}.
\newblock Cambridge Univ. Press, Cambridge, 1980.

\bibitem[DDF21]{DDF21}
Nicolas~Daans, Philip~Dittmann, and Arno~Fehm.
\newblock Existential rank and essential dimension of diophantine sets.
\newblock arXiv:2102.06941v5, 2025.

\bibitem[Den75]{Denef75}
Jan~Denef.
\newblock Hilbert's tenth problem for quadratic rings.
\newblock {\em Proc. Amer. Math. Soc.}, 48:214--220, 1975.

\bibitem[DL26]{DingLi26}
Yuchen~Ding and Junfeng~Li.
\newblock An AI proof of $18$-variable undecidability for Diophantine equations over $\Z[i]$.
\newblock arXiv:2606.12776v1, 2026.

\bibitem[HJ23]{HarringtonJones23}
Joshua~Harrington and Lenny~Jones.
\newblock A note on generalised Wall--Sun--Sun primes.
\newblock {\em Bull. Aust. Math. Soc.}, 108(3):373--378, 2023.

\bibitem[Kha22]{Kha22}
Sudesh~Kaur~Khanduja.
\newblock {\em A textbook of algebraic number theory}, {\em Unitext}, vol.~135.
\newblock Springer, Singapore, 2022.

\bibitem[KP26]{KP26}
Peter~Koymans and Carlo~Pagano.
\newblock Hilbert's tenth problem via additive combinatorics.
\newblock {\em J. Amer. Math. Soc.}, 2026. Published online.

\bibitem[MS27]{MS27}
Yuri~Matiyasevich and Zhi-Wei~Sun.
\newblock Undecidability on Diophantine equations over $\Z[i]$ with $20$ unknowns.
\newblock {\em J. Number Theory}, 290:210--219, 2027.

\bibitem[San16]{Sanna16}
Carlo~Sanna.
\newblock The $p$-adic valuation of Lucas sequences.
\newblock {\em Fibonacci Quart.}, 54(2):118--124, 2016.

\bibitem[Sie70]{Sierpinski70}
Wac{\l}aw~Sierpi\'nski.
\newblock {\em 250 Problems in Elementary Number Theory}.
\newblock American Elsevier Publishing Co., New York, 1970.

\bibitem[SS92]{SS92}
Zhi-Hong~Sun and Zhi-Wei~Sun.
\newblock Fibonacci numbers and Fermat's last theorem.
\newblock {\em Acta Arith.}, 60(4):371--388, 1992.

\bibitem[Sun21]{Sun21}
Zhi-Wei~Sun.
\newblock Further results on Hilbert's tenth problem.
\newblock {\em Sci. China Math.}, 64(2):281--306, 2021.

\bibitem[Sun26]{Sun26}
Zhi-Wei~Sun.
\newblock On Diophantine equations over the integer rings of quadratic fields.
\newblock arXiv:2608.03992v1, 2026.

\bibitem[Vin63]{Vinson63}
John~Vinson.
\newblock The relation of the period modulo $m$ to the rank of apparition of $m$ in the Fibonacci sequence.
\newblock {\em Fibonacci Quart.}, 1(2):37--45, 1963.

\bibitem[Wal60]{Wall60}
D.~D.~Wall.
\newblock Fibonacci series modulo $m$.
\newblock {\em Amer. Math. Monthly}, 67(6):525--532, 1960.

\end{thebibliography}
\end{document}